\documentclass[a4paper,11pt]{article}
\usepackage{amsmath}
\usepackage{amssymb}
\usepackage{theorem}
\usepackage{pstricks}
\usepackage{euscript}
\usepackage{epic,eepic}
\usepackage{graphicx}
\PassOptionsToPackage{normalem}{ulem}
\usepackage{amsmath}
\usepackage{amssymb}
\usepackage{theorem}
\usepackage{pstricks}
\usepackage{euscript}
\usepackage{epic,eepic}
\usepackage{graphicx}
\usepackage{setspace}
\PassOptionsToPackage{normalem}{ulem}

\newcommand{\menge}[2]{\big\{{#1} \;|\; {#2}\big\}} 

\newcommand{\Pair}[2]{{\big\langle{{#1},{#2}}\big\rangle}}

\newcommand{\emp}{\ensuremath{{\varnothing}}}

\newcommand{\scal}[2]{\left\langle{#1}\mid {#2} \right\rangle}

\newcommand{\vuo}{\ensuremath{\mbox{\footnotesize$\square$}}}

\newcommand{\YY}{\ensuremath{\mathcal Y}}
\newcommand{\XX}{\ensuremath{\mathcal X}}
\newcommand{\pair}[2]{{\langle{{#1},{#2}}\rangle}}

\newcommand{\EE}{\ensuremath{\mathsf{E}}}

\newcommand{\RR}{\ensuremath{\mathbb R}}

\newcommand{\RX}{\ensuremath{\,\left]-\infty,+\infty\right]}}
\newcommand{\NN}{\ensuremath{\mathbb N}}

\newcommand{\dom}{\ensuremath{\operatorname{dom}}}

\newcommand{\prox}{\ensuremath{\operatorname{prox}}}

\newcommand{\inte}{\ensuremath{\operatorname{int}}}

\newcommand{\argmin}{\ensuremath{\operatorname{argmin}}}

\newcommand{\xx}{\ensuremath{\mathsf{x}}}

\newcommand{\yy}{\ensuremath{\mathsf{y}}}

\newcommand{\rr}{r}

\newcommand{\E}{\ensuremath{\mathsf{E}}}

\newcommand{\pinf}{\ensuremath{+\infty}}

\newtheorem{theorem}{Theorem}[section]
\newtheorem{lemma}[theorem]{Lemma}

\newtheorem{corollary}[theorem]{Corollary}

\newtheorem{definition}[theorem]{Definition}
\theoremstyle{plain}{\theorembodyfont{\rmfamily}
\newtheorem{assumption}[theorem]{Assumption}}
\theoremstyle{plain}{\theorembodyfont{\rmfamily}
}
\theoremstyle{plain}{\theorembodyfont{\rmfamily}
\newtheorem{algorithm}[theorem]{Algorithm}}
\theoremstyle{plain}{\theorembodyfont{\rmfamily}
}
\theoremstyle{plain}{\theorembodyfont{\rmfamily}
\newtheorem{problem}[theorem]{Problem}}
\theoremstyle{plain}{\theorembodyfont{\rmfamily}
\newtheorem{remark}[theorem]{Remark}}
\theoremstyle{plain}{\theorembodyfont{\rmfamily}
}

\definecolor{labelkey}{rgb}{0,0.08,0.45}
\definecolor{refkey}{rgb}{0,0.6,0.0}
\definecolor{Brown}{rgb}{0.45,0.0,0.05}
\definecolor{dgreen}{rgb}{0.00,0.49,0.00}
\definecolor{dblue}{rgb}{0,0.08,0.75}
\numberwithin{equation}{section}

\title{A Stochastic Variance Reduction Algorithm with Bregman Distances for Structured Composite Problems }

\author{ Nguyen Van Dung$^1$ and 
B$\grave{\text{\u{a}}}$ng C\^ong V\~u$^2$\\[5mm]
 \\
$^1$ Department of Mathematics, University of Transport and Communications,\\ 3 Cau Giay Street, Hanoi, Vietnam\\
 dungnv@utc.edu.vn;  bangcvvn@gmail.com
 } 
\begin{document}
\maketitle

\begin{abstract} We develop a novel stochastic primal dual splitting method with Bregman distances for solving a structured composite problems involving infimal convolutions in non-Euclidean spaces. The sublinear convergence in expectation of the primal-dual gap is proved under mild conditions on stepsize for the general case. The linear convergence rate is obtained under additional condition like the strong convexity relative to Bregman functions. 

\end{abstract}



\section{Introduction}

The stochastic gradient method is one of the most popular algorithm in machine learning. This method  was introduced over 60 years ago.  The main advantage of this type of method is  the low computational cost per iteration. However, stochastic gradient algorithm  converges slowly and achieves only the sublinear convergence rates even when the underlying problem is strongly convex.  To accelerate the convergence of the stochastic gradient algorithms, one of the popular approach is to use the variance reduction techniques. As a result,
various variance-reduced stochastic gradient algorithms have been proposed in the literature; see \cite{bach2,Xiao14,Nitanda,Hazan} and the references therein. Recently, variance reduced methods are also investigated extensively for solving saddle point problems \cite{Bach1,Du19,Devraj19,Arizona,Shi17}. In this paper, we focus on the following stochastic optimization problem in non-Euclidean spaces.

\begin{problem}\label{prob:1} Let $d$ and $p$ be strictly positive integer.
Let $\XX$ be  $(\RR^d, \|\cdot\|_{\XX})$, let $\YY$ be $(\RR^p, \|\cdot\|_{\YY})$. 
Let $f\in\Gamma_0(\XX)$, $g\in\Gamma_0(\YY)$ and
$h\colon\XX\to\RR$; $\ell\colon \YY^* \to \RR$ be convex differentiable functions. Let $K\colon\XX\to\YY$ be a bounded linear operator. 
The primal problem is to 
\begin{align}
\underset{  x\in \XX }{\text{minimize}} \; h(x)+(\ell^*\vuo g)(Kx)+f(x), \label{e:primalap}
\end{align}
and the dual problem is to 
\begin{align}
\underset{  v\in \YY^* }{\text{minimize}} \; (h+f)^*(-K^*v)+g^*(v) +\ell(v), \label{e:dualap}
\end{align}
where $h$ and $\ell$ are given by finite sums. Here the symbol
$\vuo$ denotes the infimal convolution:
\begin{equation}
    \ell^*\vuo g\colon x\mapsto \inf_{y\in\YY} (\ell^*(y) +g(x-y)).
\end{equation}
 \end{problem}
 This framework covers various problem formulations in image processing, machine learning, statistical learning, game theory, reinforcement learning, portfiolo optimization, \cite{bach2,Xiao14,Nitanda,Hazan,Bach1,Du19,Devraj19,Arizona,Shi17,plc6}.
 
 When $\XX$ and $\YY$ are Euclidean spaces,
 this typical primal-dual framework was first investigated in \cite{plc6} and then in \cite{Vu2013,Buicomb,Bot2013,Bot2014,Vu19} for instances. In these works, various deterministic primal-dual splitting methods were proposed where the proximity operators of $f$, $g$ are often used in the backward steps, the gradient of $h$, $\ell^*$ and $K, K^*$ are often used in the forward steps. In other words, they are full-splitting. The stochastic counterparts of some  primal-dual splitting methods were also investigated  in the literature; see \cite{bang1,Dung20,Com02,bang4} for instances. 
 
 When $\XX$ and $\YY$ are non-Euclidean spaces, finding the closed-form expression of proximity operators is quite a challenge. Alternatively, the notion of the Bregman proximity operator appeared naturally and based on Bregman distance \cite{Bregman67}. This concept plays an important role in various field of applied mathematics and information theory. Numerical optimization methods with Bregman proximity operator are investigated widely in the literature; see \cite{Chen93,Sico03,Buicomb1,Dang,Duchi10,Hien16,Lei17,Nem09,Tseng08,Lu18,Quang} for instances. In particular, the stochastic numerical methods with Bregman  distance for saddle point problems are of great interest; see  \cite{Bach1,Nem09,Shi17,Chen14} and the references therein.

 The objective of this paper is to develop a new stochastic  primal-dual splitting method,
 for solving Problem \ref{prob:1}, which incorporates the following features:
 (i) Using the the Bregman proximity operators which allows us the implement the proposed algorithm easily in non-Euclidan spaces; (ii) Using the acceleration technique in term of variance reduced to obtain a faster convergence rate; (iii) The proposed algorithm is full splitting.  To the best of our knowledge, there exists no method in the literature for solving Problem \ref{prob:1} with non-trivial $\ell$, which has simultaneously these features.
 
 The paper is organized as follows. In Section 2, we recall the basic notions in convex analysis and we precisely state our assumptions which will be used in the proof of the convergence of the proposed method. Subsequently, in Section 3, we present the proposed methods. Then, we derive the linear convergence in expectation of the primal-dual gap for the strong convex case, 
 and the sublinear convergence for the general case.

\noindent{\bf Notations.} Throughout this paper, we use the notation $\|\cdot\|$ for any norm in the spaces $\XX,\YY$ as well as their dual spaces $\XX^*,\YY^*$. The conjugate of the operator $K$ 
is denoted by $K^*$.
The interior and closure  of a subset $X$ is denoted by $\inte(X)$ and $\overline{X}$, respectively.
The domain of a function 
$f\colon \XX\to \left]0,+\infty\right]$ is 
$\dom(f)=\menge{x\in\XX}{f(x)< +\infty}$. This function is proper if $\dom(f)\not=\emp$.
We denote
$\Gamma_0(\XX)$
the class of all proper lower semicontinous functions $f$ from $\XX$ to $\left]0,+\infty\right]$. Let $A$ be a set-valued operator on $\XX$, the inverse of $A$ is 
$A^{-1}\colon x\mapsto \menge{u\in \XX}{x\in Au}$.
The expectation of a random variable $x$ is denoted by $\E\left[x\right]$.   
 
\section{Preliminaries}
\subsection{Bregman distance}
We first recall the following definitions.

\begin{definition}{\rm\cite{Ccm01,Sico03}}
	\label{def:1}
Let $\phi \in\Gamma_0(\XX)$. Its conjutate is 
 $\phi^*\colon\XX^*\to\RX\colon
x^*\mapsto\sup_{x\in\XX}(\pair{x}{x^*}-\phi(x))$ and its
Moreau subdifferential \cite{Mor63c} is
\begin{equation}
\label{e:subdiff}
\partial\phi\colon\XX\to 2^{\XX^*}\colon x\mapsto
\menge{x^*\in\XX^*}{(\forall y\in\XX)\,
\Pair{y-x}{x^*}+\phi(x)\leq\phi(y)}.  
\end{equation}
Recall that $\phi$ is a Legendre function if it is 
essentially smooth in the sense that $\partial\phi$ is 
both locally bounded and single-valued on its
domain, and essentially strictly convex in the sense
that $\partial\phi^*$ is locally bounded on its domain and 
$\phi$ is strictly convex on every convex subset of  $\dom\partial\phi$.
Moreover, $\phi$ is G\^ateaux differentiable on $\inte\dom\phi\neq\emp$
and the associated \emph{Bregman distance} is defined by
\begin{equation}
\label{e:Bdist}
\begin{aligned}
D_{\phi}\colon\XX\times\XX&\to\,[0,+\infty]\\
(x,y)&\mapsto 
			\begin{cases}
				\phi(x)-\phi(y)-\Pair{x-y}{\nabla\phi(y)},&\text{if}\;\;y\in\inte\dom\phi;\\
				\pinf,&\text{otherwise},
			\end{cases}
\end{aligned}
\end{equation}
and the Bregman proximity operator of $g\in\Gamma_0(\XX)$ with respect to $\phi$ is
\begin{equation}
\label{eq121d}
\prox_{g}^{\phi}\colon\XX\to 2^{\XX}\colon x\mapsto \arg\min_{y\in\XX} g(y) +D_{\phi}(y,x).
\end{equation}
\end{definition}
Note that $\prox_{g}^{\phi}$ is well-defined and single valued,		$$\prox_{g}^{\phi}=(\nabla\phi+\partial g)^{-1}\circ\nabla\phi\colon\inte\dom\phi\to\inte\dom\phi.
$$
Let us recall 
the following lemma play a key role in the analysis of Bregman-based methods.

		

\begin{lemma}
	\label{lem123:1}
	Let $x\in\XX$ and let $(z,p)\in(\inte\dom\phi)^2$. The following hold.
	\begin{enumerate}
		\item
		\label{lem123:1i}
		$\Pair{x-p}{\nabla\phi(z)-\nabla\phi(p)}=D(x,p)+D(p,z)-D(x,z)$.
		\item
		\label{lem123:1ib}
		$\Pair{z-p}{\nabla\phi(z)-\nabla\phi(p)}=D(z,p)+D(p,z)$.
		\item
		\label{lem123:1ii}
		Suppose that $p = \prox_{g}^{\phi}z$ and $g$ is $\alpha$-strongly convex relative to $\phi$, then
		\begin{equation}
			g(p) + D(p,z) \leq g(x) + D(x,z) - (1+\alpha)D(x,p).
		\end{equation}
	\end{enumerate}
\end{lemma}

\begin{proof}
	\ref{lem123:1i}\&\ref{lem123:1ib}: \cite[Proposition~2.3]{Sico03}.
	
	\ref{lem123:1ii}: By definition, 
	\begin{equation}
		\nabla\phi(z)-\nabla\phi(p)\in\partial g(p),
	\end{equation}
	and hence \eqref{e:subdiff} yields
	\begin{equation}
	\alpha D(x,p)+	g(p)+\Pair{x-p}{\nabla\phi(z)-\nabla\phi(p)}\leq g(x).
	\end{equation}
	Consequently, the assertion follows from \ref{lem123:1i}.
\end{proof}


%
\subsection{Assumptions}
We state several additional assumptions to Problem \ref{prob:1} which will be used in the next section. 
\begin{assumption}\label{as:1}
 There exists a point $(x^\star,v^\star)\in\XX\times\YY^*$ such that the primal-dual gap function defined by
	\begin{align}
	G:&\mathcal{X}\times\mathcal{Y}^* \to \mathbb R \cup \{-\infty,+\infty\} \notag\\
	&(x,v) \mapsto h(x)+f(x)+\langle Kx,v\rangle -g^*(v) -\ell(v) \label{s}
	\end{align}
	verifies the following condition:
	\begin{align}\label{e:saddle}
	\big(\forall x\in\XX\big)\big(\forall v\in\YY^*\big)\;
	G(x^\star,v) \leq G(x^\star,v^\star) \leq G(x,v^{\star}),
	\end{align}
	We denote $\mathcal{S}$ the set of all point $(x^*,v^*)$ such that \eqref{e:saddle} is satisfied.
\end{assumption}
\begin{assumption}\label{giathiet}\upshape  The following conditions will be used.
	\begin{enumerate}
		\item [\textup{(B0)}] $\varphi$ and $\psi$ are $1$-strong convex,  G\^ateaux differentiable on $\inte\dom\varphi\neq\emp$ and  $\inte\dom\psi\neq\emp$, ,
		\emph{Legendre functions} on $\XX$ and 
		$\YY^*$, respectively. Moreover, $\dom(f) \subset \inte \dom(\varphi)$ and 
		$\dom(g^*)\subset \inte \dom(\psi)$. 
		\item[\textup{(B1)}] $h=\dfrac{1}{n} \sum \limits_{i=1}^n h_i$ 
		and $\ell =\dfrac{1}{n'} \sum \limits_{i=1}^{n'} \ell_i$,
		where $h_i$ and $\ell_i$ are differentiable with  $\mu_i$ and $\nu_i$-Lipschitz continuous, respectively: $$\|\nabla h_i(x)-\nabla h_i(y)\| \le \mu_i \|x-y\|\; \text{and}\; \|\nabla \ell_i(u)-\nabla \ell_i(v)\| \le \nu_i \|u-v\|.$$
		Hence,
		$$\| \nabla h(x)-\nabla h(y) \| \le \mu \|x-y\|, \text{	where $\mu =(1/n) \sum \limits_{i=1}^n \mu_i,$}$$
	and 
		$$\| \nabla \ell(u)-\nabla \ell(v) \| \le \nu \|u-v\|,\;
		\text{where $\nu =(1/n') \sum \limits_{j=1}^{n'} \nu_j$.}
		$$
		
		\end{enumerate}
\end{assumption}
\begin{assumption}\label{giathiet3}
The functions $f$ and $g^*$ are 
			$\alpha$-strongly convex relative to $\varphi$ and $\psi$, respectively, i.e.,
			\begin{equation}
              \begin{cases}
       (\forall x,\ y \in \inte\dom(\varphi))\ (\forall p\in\partial f(y))\;f(x) \geq f(y) + \pair{p}{x-y} + \alpha D_{\varphi}(x,y),\\
       (\forall u,\ v\in\inte\dom(\psi))\ (\forall q\in\partial g^*(v))\    g^*(u) \geq g^*(v) + \pair{q}{u-v} + \alpha D_{\psi}(u,v).
              \end{cases}
			\end{equation}
\end{assumption}
\begin{remark} Assumption \ref{as:1} and Assumption \ref{giathiet} are very popular in the literature. While, many  of functions, for which  Assumption \ref{giathiet3} can be verified, can be found  in \cite{Duchi10,Shi17,Lu18}. 
\end{remark}
\section{Main results}\label{sec:Alg}
We propose the following primal dual splitting method 
for solving Problem \ref{prob:1} which incorporates the variance reduction technique as the Bregman proximity operators at each iteration. 

\begin{algorithm}
	\label{alg3}
	\vspace*{0.3em}
	Let $(\bar x_0,\bar v_0)\in\inte(\dom(\varphi))\times\inte(\dom(\psi))$, let $\gamma>0,\ \theta \in \{0;1\}$ and $m$ is a strictly positive integer. Let $(\omega_k)_{k\in\NN}$ be a strictly positive sequence in $\RR$.	Let  $Q=\{q_1, \ldots,q_n\}$ and $Q'=\{q'_1, \ldots, q'_{n'} \} $ be the probabilites on  $\{1,\ldots,n\}$  and  $ \{1,\ldots,n' \}$, respectively.
	\\
	Set $x^0_0=x^0_{-1}=\bar x_0,\ v^0_0=v^0_{-1}= \bar v_0.	$\\
	{\bf Iterate:} for $s=0,1,2,\ldots$
	\begin{align*}
	&\bar x=\bar x_s, \bar v=\bar v_s\\
	&x_0 = x_{0}^s,\ x_{-1}= x_{-1}^s\\
	&v_0 = v_{0}^s,\ v_{-1}= v_{-1}^s\\
	&{\bf iterate:}\ \text{for}\ k=0,1,\ldots,m-1\\
	&\hspace{1cm} \text{pick} \ i_k \in \{1,\ldots,n\} \ \text{and} \  \ j_{k} \in \{1,\ldots,n' \} \ \text{randomly according to $Q$ and $Q'$, respectively}\\
	&\hspace{1cm}\begin{cases}
	y_k&= x_k+\theta(x_k -x_{k-1})\\
	z_k&= \dfrac{ \nabla h_{i_k}(y_k)- \nabla h_{i_k}(\bar x)}{q_{i_k}n}+\nabla h(\bar x)\\
	u_k&=v_k+\theta(v_k-v_{k-1})\\
	t_k&= \dfrac{\nabla \ell_{j_k}(u_k)-\nabla \ell_{j_k}(\bar v)}{q'_{j_k}n'}+\nabla \ell(\bar v)\\
	x_{k+1}&=(\nabla \varphi + \gamma \partial f)^{-1}(\nabla \varphi(x_k)-\gamma z_k-\gamma K^* u_k)\\
	v_{k+1}&=(\nabla \psi + \gamma \partial g^*)^{-1}(\nabla \psi(v_k)-\gamma t_k+\gamma K y_k),
	\end{cases}\\
	&\text{\bf end}\\
		& \text{set}\ \bar x_{s+1}= \sum_{k=1}^m \omega_k x_k\; \text{and}\; \bar v_{s+1}=\sum_{k=1}^m \omega_k v_k.\\
		& x^{s+1}_0=x_m,\ x^{s+1}_{-1}=x_{m-1},\ v^{s+1}_0=v_m,\ v^{s+1}_{-1}=v_{m-1}.
		\end{align*}
		{\bf end}
\end{algorithm}
\begin{remark} Here are some remarks.
\begin{enumerate}
    \item Algorithm \ref{alg3} is an extension of the one in \cite{Dung20} which is restricted to Hilbert spaces setting with the Euclidean norm
and $\theta= 1$ and without variance reduction. Further connections to existing work \cite{sva2} can be found in \cite{Dung20}. In particular, we use the variance reduction technique as in \cite{Xiao14,Shi17}.
 When $\theta\not=0$,
the proposed algorithm is different from the inertial stochastic primal-dual splitting method in \cite{bang1}.
\item Recently, there appeared several stochastic variance reduction algorithms in   \cite{Bach1,Nem09,Shi17,Chen14} which can be used to solve Problem \ref{prob:1}. However, the resulting algorithms are not full-splitting and they are different from our proposed algorithm here.
\end{enumerate}
\end{remark}
Before stating our main convergence results of Algorithm \ref{alg3}, we first prove several auxiliary results which are natural extension of 
\cite[Lemma 1, Corollary 3]{Xiao14} from primal framework to our primal-dual framework.

\begin{lemma}
\label{lm1dad}
Suppose the Assumption \ref{giathiet} $(B1)$ and Assumption \ref{as:1} are satisfied. Then, for all $(x,v)\in \dom(f)\times\dom(g^*)$ and for all $(x^*,v^*)\in\mathcal{S}$, we have
\begin{equation} \label{e:sd}
\dfrac{1}{n} \sum \limits_{i=1}^n \dfrac{1}{nq_i} \|\nabla h_i(x)-\nabla h_i(x^\star)\|^2 \le 2L_Q [G(x,v^\star)-G(x^\star,v^\star)],
\end{equation}
and 
\begin{equation}\label{e:sc}
\dfrac{1}{n'} \sum \limits_{j=1}^{n'} \dfrac{1}{n'q'_j} \|\nabla \ell_j(v)-\nabla \ell_j(v^\star)\|^2 \le 2L_{Q'} [G(x^\star,v^\star)-G(x^\star,v)],
\end{equation}
where $L_Q=\max_i \mu_i/(q_i n),\ L_{Q'}=\max_j \nu_j/(q_j' n').$
\end{lemma}

\begin{proof}
    It follows from \cite[Lemma 2]{Shi17} that
    $$ \dfrac{1}{n} \sum \limits_{i=1}^n \dfrac{1}{nq_i} \|\nabla h_i(x)-\nabla h_i(x^\star)\|^2 \le 2L_Q[h(x)-h(x^\star)-\pair{x-x^\star}{\nabla h(x^\star)}.$$
    Since $(x^*,v^*)\in \mathcal{S}$, we have 
    $$x^\star=\underset{x}{\argmin}\ G(x,v^\star)=\underset{x}{\argmin}\  \{h(x)+R(x)\},$$
    where $R(x)=f(x)+\pair{Kx}{v^\star}-g^*(v^\star)-\ell(v^\star).$ By the optimality of $x^\star$,
    there exists $\xi^\star \in \partial R(x^\star)$ such that $\nabla h(x^\star)+\xi^\star=0.$ Therefore
    \begin{align*}
        h(x)-h(x^\star)-\pair{x-x^\star}{\nabla h(x^\star)}&=h(x)-h(x^\star)+\pair{x-x^\star}{\xi^\star}\\
        &\le h(x)-h(x^\star)+R(x)-R(x^\star)\\
        &=G(x,v^\star)-G(x^\star,v^\star),
    \end{align*}
    where the second inequality follows from the convexity of $R$. Hence, \eqref{e:sd} is proved.
    The estimation \eqref{e:sc}
    is proved by the same fashion.
\end{proof}
\begin{corollary}\label{lm1}
Under the same assumptions as in Lemma \ref{lm1dad}. Let
 $(x_k)_{k\in\NN}, (y_k)_{k\in\NN},$\\
 $(u_k)_{k\in\NN}, (v_k)_{k\in\NN}$, $ (z_k)_{k\in\NN}$ $(t_k)_{k\in\NN}$ be sequences generated by Algorithm \ref{alg3}. Let
 $\E_{i_k}$  and  $\E_{j_k}$  be the conditional expectation with respect to the history $\{(i_0,j_0),\ldots, (i_{k-1},j_{k-1})\}$. Then,
 we have
 \begin{equation}\label{e:unbiased1}
 (\forall k\in\NN)\;
 \E_{i_k}\left[z_k\right]=\nabla h(y_k)\quad \text{and}\quad \E_{j_k}\left[t_k\right]=\nabla \ell (u_k) .
 \end{equation}
 Moreover, set $L_1=\max \{L_Q,\ L_{Q'}\},\ L_2=\max \{ \mu_i^2/(q_i n),\  \nu_j^2/(q_j'n')\} $.
 If $\theta=0$ then 
 \begin{equation}\label{e:var1}
 \begin{cases}
  \E_{i_k} \|z_k- \nabla h(y_k)\|^2 &\le 4L_1\big( [G(x_k,v^\star)-G(x^\star,v^\star)]+G(\bar x,v^\star)-G(x^\star,v^\star) \big)\\
  \E_{j_k} \|t_k-\nabla \ell (u_k)\|^2 &\le 4L_1 \big( [G(x^\star,v^\star)-G(x^\star,v_k)]+G( x^\star,v^\star)-G(x^\star,\bar v)\big).
 \end{cases}
 \end{equation}
If $\theta=1$ then 
 \begin{equation}\label{e:var2}
 \begin{cases}
 \E_{i_k} \|z_k-\nabla h (y_k)\|^2 &\le 4 L_2  \|x_k-x_{k-1}\|^2+8L_1  [G(x_k,v^\star)-G(x^\star,v^\star)] \\
 &\ \ +4L_1 [G( \bar x,v^\star)-G(x^\star, v^\star))\\
 \E_{j_k} \|t_k-\nabla \ell (u_k)\|^2 &\le 4 L_2 \|v_k-v_{k-1}\|^2+8L_1  [G(x^\star,v^\star)-G(x^\star,v_k)] \\
 &\ \ +4L_1 [G( x^\star,v^\star)-G(x^\star,\bar v) ).
  \end{cases}
 \end{equation}
\end{corollary}
\begin{proof}
We take expectation with respect to $i_k$ to obtain
	$$\E_{i_k} \big[\dfrac{1}{nq_{i_k}}\nabla h_{i_k}(y_k)\big]=\sum_{i=1}^n \dfrac{q_i}{nq_i} \nabla h_i(y_k)=\sum_{i=1}^n \dfrac{1}{n}\nabla h_i(y_k)=\nabla h(y_k).$$
	Similarly, we have $\E_{i_k}[(1/(nq_{i_k}))\nabla h_{i_k}(\bar x)]=\nabla h(\bar x)$. Therefore
	$$\E_{i_k}\left[ z_k\right]=\E_{i_k} \big[ \dfrac{ \nabla h_{i_k}(y_k)- \nabla h_{i_k}(\bar x)}{q_{i_k}n}+\nabla h(\bar x)\big]=\nabla h(y_k).$$
	Using the same argument, we also obtain $\E_{j_k} \left[t_k\right]=\nabla \ell (u_k).$ Hence, \eqref{e:unbiased1} is proved. We next
	bound the variance, we have
	\begin{align*}
	\E_{i_k}\left[\|z_k-\nabla h(y_k)\|^2\right]&=\E_{i_k} \left[\big\|\dfrac{1}{nq_{i_k}}\big( \nabla h_{i_k}(y_k)-\nabla h_{i_k}(\bar x)\big)+\nabla h(\bar x)-\nabla h(y_k) \big\|^2\right]\\	
	&=\E_{i_k} \left[\dfrac{1}{(nq_{i_k})^2}\|\nabla h_{i_k}(y_k)-\nabla h_{i_k}(\bar x)\|^2-\|\nabla h(y_k)-\nabla h(\bar x)\|^2\right]\\
	&\le \E_{i_k}\left[ \dfrac{1}{(nq_{i_k})^2}\|\nabla h_{i_k}(y_k)-\nabla h_{i_k}(\bar x)\|^2\right]\\
	&=\dfrac{1}{n} \sum_{i=1}^n \dfrac{1}{nq_i}\|\nabla h_i(y_k)-\nabla h_i(\bar x)\|^2\\
	&\le \dfrac{2}{n} \sum_{i=1}^n \dfrac{1}{nq_i}\big(\|\nabla h_i(y_k)-\nabla h_i( x^\star)\|^2+\|\nabla h_i(\bar x)-\nabla h_i(x^\star)\|^2\big).\\
	\end{align*}
	If $\theta=0$ then $y_k=x_k$.  Using Lemma \ref{lm1dad}, we derive that
	\begin{align*}
	   	\E_{i_k}\left[\|z_k-\nabla h(y_k)\|^2\right] & \le \dfrac{2}{n} \sum_{i=1}^n \dfrac{1}{nq_i}\big(\|\nabla h_i(x_k)-\nabla h_i( x^\star)\|^2+\|\nabla h_i(\bar x)-\nabla h_i(x^\star)\|^2\big) \\
	   	& \le 4L_1\big( [G(x_k,v^\star)-G(x^\star,v^\star)]+G(\bar x,v^\star)-G(x^\star,v^\star) \big).
	\end{align*}
	By the same manner, we also have
	\begin{equation*}
	    \E_{j_k} \|t_k-\nabla \ell (u_k)\|^2 \le 4L_1 \big( [G(x^\star,v^\star)-G(x^\star,v_k)]+G( x^\star,v^\star)-G(x^\star,\bar v)\big).
	\end{equation*}
	Therefore, \eqref{e:var1} is proved. Now, let us consider the case $\theta =1$. We have
	\begin{align*}\E_{i_k}\left[\|z_k-\nabla h(y_k)\|^2\right]	&\le \dfrac{4}{n} \sum_{i=1}^n \dfrac{1}{nq_i} \big( \|\nabla h_i(y_k)-\nabla h_i(x_k)\|^2+ \|\nabla h_i(x_k)-\nabla h_i(x^\star)\|^2\big)\\
	&\ +\dfrac{2}{n} \sum_{i=1}^n \dfrac{1}{nq_i}\|\nabla h_i(\bar x)-\nabla h_i(x^\star)\|^2\\
	&\le \dfrac{4}{n} \sum_{i=1}^n \dfrac{\mu_i^2}{nq_i}\|x_k-x_{k-1}\|^2+8L_Q [G(x_k,v^\star)-G(x^\star,v^\star)]\\
	&\ \ +4L_Q[G(\bar x,v^\star)-G(x^\star,v^\star)]\\
	&\le 4 L_2  \|x_k-x_{k-1}\|^2+8L_1 [G(x_k,v^\star)-G(x^\star,v^\star)]\\
	&\ \ +4L_1[G(\bar x,v^\star)-G(x^\star,v^\star)].
	\end{align*}
	Similarly, we also have 
	\begin{equation*}
	    \E_{j_k}\left[ \|t_k-\nabla \ell (u_k)\|^2\right] \le 4 L_2  \|v_k-v_{k-1}\|^2+8L_1  [G(x^\star,v^\star)-G(x^\star,v_k)]+4L_1 [G( x^\star,v^\star)-G(x^\star,\bar v) ).
	\end{equation*}
	Hence, the proof is completed.
\end{proof}
\begin{lemma}\label{vietlai}
Suppose that Assumption \ref{giathiet} is satisfied. Let
 $(x_k)_{k\in\NN}, (y_k)_{k\in\NN},$
 $(u_k)_{k\in\NN}, (v_k)_{k\in\NN}$, $ (z_k)_{k\in\NN}$ $(t_k)_{k\in\NN}$ be sequences generated by Algorithm \ref{alg3} at the stage $s$. Let $	\xx^*=(x^*,v^*) \in \mathcal{S}$. Define
 	\begin{equation}
 	\big(\forall k\in\{0,\ldots m\}\big)
	\begin{cases}
	\xx_k &= (x_k,v_k),
	\yy_k = (y_k,u_k),\; \hat{\xx}_k= (\hat{x}_k,\hat{v}_k),\\
	\rr_k &= (z_k,t_k ),\\
	\mathsf{R}_k &= (\nabla h(y_k), \nabla \ell(u_k)),\\
		L\colon&\XX\times\YY^*\to\XX^*\times\YY\colon (x,v)\mapsto (K^*v,-Kx),\\
			    \varphi\oplus\psi\colon& (x,v)\mapsto \varphi(x)+\psi(v),\\
			    b_k &=\pair{L(\xx_{k}-\xx_{k-1})}{\xx_{k}-\xx^\star}.
\end{cases}
	\end{equation}
	Set $\mu_0=\max\ \{\mu,\nu \}$ and $ D= D_{\varphi\oplus\psi}.$
The following hold for any strictly positive $M$:
\begin{enumerate}
    \item\label{vietlai:i} If $\theta= 0$, then
    \begin{align}
    G(x_{k+1},v^*)-G(x^*,v_{k+1}) &\le 2\|K\| \big( M D(\xx_{k+1},\xx_k)+\dfrac{D(\xx^*,\xx_{k+1})}{M} \big)
    -\alpha D(\xx^*,\xx_{k+1}) \notag \\
    &\ +\dfrac{1}{\gamma} \big(D(\xx^*,\xx_k) - D(\xx^*,\xx_{k+1}) -D(\xx_{k+1},\xx_k) \big)  \notag \\
    &\ + \mu_0 D(\xx_{k+1},\xx_{k})  +\gamma \|\rr_k-\mathsf{R}_k\|^2+ \pair{\hat \xx_{k+1}-\xx^*}{\mathsf{R}_k-\rr_k}.\label{e:suachua1}
    \end{align}
    \item\label{vietlai:ii} If $ \theta=1$, then
    \begin{align}
 \gamma [ G(x_{k+1},v^\star)-G(x^\star,v_{k+1})] & \le D(\xx^\star,\xx_k) -\gamma b_k-D(\xx^\star,\xx_{k+1})+\gamma b_{k+1}\notag \\
    &-(1-2\gamma \mu_0-\gamma \|K\|)D(\xx_{k+1},\xx_k) \notag \\
    &+(2\gamma \mu_0 +\gamma \|K\|)D(\xx_{k},\xx_{k-1}) \notag \\
    &+\gamma^2 \|\rr_k-\mathsf{R}_k\|^2+\gamma \pair{\hat \xx_{k+1}-\xx}{\mathsf{R}_k-\rr_k}.\label{e:suachua2}
\end{align}
\end{enumerate}
\end{lemma}
\begin{proof} Let $ k\in\{0,\ldots, m-1\}$. 
 We have
		$v_{k+1}=(\nabla \psi + \gamma \partial g^*)^{-1}(\nabla \psi(v_k)-\gamma t_k+\gamma K y_k)$, which is equivalent to
		$$K y_k-t_k+\dfrac{1}{\gamma} (\nabla \psi(v_k)-\nabla \psi(v_{k+1})) \in \partial g^*(v_{k+1}).$$
	Since $g^*$ is $\alpha$-strongly convex relative to $\psi$, which implies that
		$$ g^*(v) \ge g^*(v_{k+1})+ \pair{Ky_k-t_k+\dfrac{1}{\gamma} (\nabla \psi(v_k)-\nabla \psi(v_{k+1}))}{v-v_{k+1}} +\alpha D_{\psi}(v,v_{k+1}).$$
		Using Lemma \ref{lem123:1} \ref{lem123:1i}, we have
		\begin{align}
		    g^*(v_{k+1})-g^*(v) &\le \pair{t_k-Ky_k}{v-v_{k+1}}+\dfrac{1}{\gamma} \pair {\nabla \psi(v_k)-\nabla \psi(v_{k+1})}{v_{k+1}-v}-\alpha D_{\psi}(v,v_{k+1})\notag \\
		    &=\pair{t_k-Ky_k}{v-v_{k+1}}+\dfrac{1}{\gamma} (D_\psi(v,v_k)-D_\psi (v_{k+1},v_k)-D_\psi(v,v_{k+1}))\notag\\
		    &\quad-\alpha D_{\psi}(v,v_{k+1}). \label{d1}
		\end{align}
			Since $\ell$ is convex differentiable with 
	$\nu$-Lipschitz gradient, we have 
	\begin{equation}
	\ell(v_{k+1})-\ell(v)
	\leq \pair{v_{k+1}-v}{ \nabla \ell(u_k)} +\dfrac{\nu}{2}\|v_{k+1}-u_k\|^2. \label{d2}
	\end{equation}
	 We derive from \eqref{d1} and \eqref{d2} that
	\begin{align}
	  	G(x_{k+1},v)-&G(x_{k+1},v_{k+1})=\pair{K x_{k+1}}{v-v_{k+1}}  -g^*(v)+g^*(v_{k+1})-\ell(v)+\ell(v_{k+1}) \notag \\
	&\le \pair {K(x_{k+1}-y_k)}{v-v_{k+1}} +\dfrac{1}{\gamma}(D_{\psi}(v,v_k)-D_{\psi}(v_{k+1},v_k)-D_{\psi}(v,v_{k+1})) \notag \\
	& \ \ + \frac{\nu}{2}\|v_{k+1}-u_k\|^2 + \pair{ \nabla \ell(u_k)-t_k}{v_{k+1}-v}-\alpha D_{\psi}(v,v_{k+1}).  \label{2d}
	\end{align}
	Similar to \eqref{2d}, we have,
		\begin{align}
	G(x_{k+1},v_{k+1})-&G(x,v_{k+1})=h(x_{k+1})-h(x)+\pair{ K(x_{k+1}-x)}{v_{k+1}}+f(x_{k+1})-f(x) \notag \\
	&\le \pair{ K(x_{k+1}-x)}{v_{k+1}-u_k}+\frac{1}{\gamma} (D_{\varphi}(x,x_k)-D_{\varphi}(x_{k+1},x_k)-D_{\varphi}(x,x_{k+1})) \notag \\
	&\ \ +\dfrac{\mu}{2}\|x_{k+1}-y_k\|^2 
	+ \pair{x_{k+1}-x}{ \nabla h(y_k)-z_k}-\alpha D_{\varphi}(x,x_{k+1}). \label{2b}
	\end{align}
	Adding \eqref{2d} and \eqref{2b}, we obtain
\begin{align}
    &G(x_{k+1},v)-G(x,v_{k+1}) \le \big(\pair{K(x_{k+1}-x)}{v_{k+1}-u_k}+ \pair {K(x_{k+1}-y_k)}{v-v_{k+1}} \big)\notag \\
    &\ +\dfrac{1}{\gamma} \bigg(D_{\varphi}(x,x_k)-D_{\varphi}(x_{k+1},x_k)-D_{\varphi}(x,x_{k+1})+D_{\psi}(v,v_k)-D_{\psi}(v_{k+1},v_k)-D_{\psi}(v,v_{k+1}) \bigg) \notag \\
    &\ +\dfrac{\mu}{2}\|x_{k+1}-y_k\|^2+\dfrac{\nu}{2}\|v_{k+1}-u_k\|^2+\pair{x_{k+1}-x}{ \nabla h(y_k)-z_k}+\pair{ \nabla \ell(u_k)-t_k}{v_{k+1}-v}\notag\\
    &-\alpha D_{\varphi}(x,x_{k+1})-\alpha D_{\psi}(v,v_{k+1}).
    \label{e:gap}
\end{align}
Let us set
	\begin{equation}\label{e:mas1}
	\begin{cases}
	\hat x_{k+1}&= (\nabla \varphi + \gamma \partial f)^{-1}(\nabla \varphi(x_k)-\gamma \nabla h(y_k)-\gamma K^* u_k),\\
	\hat v_{k+1}&=(\nabla \psi + \gamma \partial g^*)^{-1}(\nabla \psi(v_k)-\gamma \nabla \ell(u_k)+\gamma K y_k).
	\end{cases}
	\end{equation}
	Then, the first equation in \eqref{e:mas1} is equivalent to 
	\begin{equation}
	     \nabla \varphi(x_k)  -\gamma \nabla h(y_k)-\gamma K^* u_k \in \gamma \partial f(\hat x_{k+1})  +\nabla\varphi(\hat x_{k+1}).\label{e:mas2}
	\end{equation}
	Since $\varphi$ is $1$-strongly convex, $\gamma \partial f+ \nabla \varphi$ is $1$-strongly monotone. Hence, it follows from \eqref{e:mas2} that
	\begin{align}
	    \|\hat x_{k+1}-x_{k+1}\|^2 &\leq \pair{\hat x_{k+1}-x_{k+1}}{ \nabla \varphi(x_k)  -\gamma \nabla h(y_k)-\gamma K^* u_k +\gamma z_k+\gamma K^* u_k -\nabla \varphi(x_k) }\notag\\
	    &= \gamma \pair{\hat x_{k+1}-x_{k+1}}{\nabla h(y_k)-z_k },
	\end{align}
	which implies that 
	\begin{equation}
	    \|\hat x_{k+1}-x_{k+1}\| \leq \gamma \|z_k-\nabla h(y_k)\|.
	\end{equation}
	In turn,
	\begin{align}
	&\pair{x_{k+1}-x}{ \nabla h(y_k)-z_k} \notag\\
	&\quad = \pair{x_{k+1}-\hat x_{k+1}}{\nabla h(y_k)-z_k}+
	\scal{\hat x_{k+1}-x }{\nabla h(y_k)-z_k} \notag \\
	&\ \ \ \le \|z_k-\nabla h(y_k)\|\|x_{k+1}-\hat x_{k+1}\|+\pair{\hat x_{k+1}-x}{ \nabla h(y_k)-z_k} \notag \\
	&\ \ \ \le  \gamma \|z_k-\nabla h(y_k)\|^2+ \pair{\hat x_{k+1}-x}{ \nabla h(y_k)-z_k}. \label{4}
	\end{align}
	By the same way,
	\begin{align}
	\pair{\nabla \ell(u_k)-t_k}{v_{k+1}-v}  \le   \gamma \|t_k-\nabla \ell(u_k)\|^2+\pair{\nabla \ell(u_k)-t_k}{\hat v_{k+1}-v}\label{4'}.
	\end{align}
	Let us define $\XX\times\YY^*$ the standard product space equipped with the norm 
	$(x,v)\mapsto \sqrt{\|x\|^2 + \|v\|^2}$.
	Then $\|L\|=\|K\|$ and $\nabla (\varphi\oplus\psi)(x,v) = (\nabla\varphi(x),\nabla \varphi(v))$. Hence, for every $(x,v)$ and $(y,w)$ in $\XX\times\YY^*$, 
	\begin{align}
	    D_{\varphi}(x,y) + D_{\psi}(v,w) &= \varphi(x) + \psi(v) - (\varphi(y)+ \psi(w)) - 
	    (\pair{x-y}{\nabla\varphi(x)} +\pair{v-w}{\nabla\psi(w)})\notag\\
	    &=  \varphi\oplus\psi (x,v) +  \varphi\oplus\psi (y,w) -\pair{(x,v)-(y,w)}{\nabla (\varphi\oplus\psi)(x,v)}\notag\\
	    &= D((x,v),(y,w)).
	\end{align} 
 In turn,  the term in the second line in \eqref{e:gap} can be rewritten as 
	\begin{align}
	    \dfrac{1}{\gamma} \big(D_{\varphi}(x,x_k)-D_{\varphi}(x_{k+1},x_k)-&D_{\varphi}(x,x_{k+1})+D_{\psi}(v,v_k)-D_{\psi}(v_{k+1},v_k)-D_{\psi}(v,v_{k+1}) \big) \notag\\
	    & =  \dfrac{1}{\gamma} \big(D_{\varphi\oplus\psi}(\xx,\xx_k) - D_{\varphi\oplus\psi}(\xx,\xx_{k+1}) -D_{\varphi\oplus\psi}(\xx_{k+1},\xx_k) \big) \notag \\
	      & =  \dfrac{1}{\gamma} \big(D(\xx,\xx_k) - D(\xx,\xx_{k+1}) -D(\xx_{k+1},\xx_k) \big)\label{e:dvd1}.
	\end{align}
	(i). Let us consider the case when
	$\theta= 0$. we can rewrite \eqref{e:gap}:
	\begin{align}
    G(x_{k+1},v)-G(x,v_{k+1}) &\le \big(\pair{K(x_{k+1}-x)}{v_{k+1}-v_k}+ \pair {K(x_{k+1}-x_k)}{v-v_{k+1}} \big)\notag \\
    &\ +\dfrac{1}{\gamma} \bigg(D(\xx,\xx_k)-D(\xx_{k+1},\xx_k)-D(\xx,\xx_{k+1}) \bigg) \notag \\
    &\ +\dfrac{\mu}{2}\|x_{k+1}-x_k\|^2+\dfrac{\nu}{2}\|v_{k+1}-v_k\|^2-\alpha D(\xx,\xx_{k+1})\notag \\
    &+\pair{x_{k+1}-x}{ \nabla h(x_k)-z_k}+\pair{ \nabla \ell(v_k)-t_k}{v_{k+1}-v}    .
    \label{e:gapd}
\end{align}
	The term in the first and the third line in \eqref{e:gapd} can be estimated:
	\begin{align}
	    \pair{K(x_{k+1}-x_k)}{v-v_{k+1}}&+\pair{K(x_{k+1}-x)}{v_{k+1}-v_k} \notag \\
	   	    & \le \|K\| \bigg(M \|\xx_{k+1}-\xx_k\|^2+\dfrac{\|\xx_{k+1}-\xx\|^2}{M} \bigg) \notag \\
	    &\le 2\|K\| \big( M D(\xx_{k+1},\xx_k)+\dfrac{D(\xx,\xx_{k+1})}{M}\big). \label{inedad9}
	\end{align}
	Since $\varphi\oplus\psi$ is $1$-strongly convex, we have
	\begin{align}
	  \dfrac{\mu }{2}\|x_{k+1}-x_k\|^2 + \frac{\nu }{2}\|v_{k+1}-v_k\|^2 &\leq \mu_0 D(\xx_{k+1},\xx_k).
	\end{align}
	Therefore, using \eqref{4} and \eqref{4'}, we derive from \eqref{e:gapd} that
		\begin{align}
    G(x_{k+1},v)-G(x,v_{k+1}) &\le 2\|K\| \big( M D(\xx_{k+1},\xx_k)+\dfrac{D(\xx,\xx_{k+1})}{M} \big)
    -\alpha D(\xx,\xx_{k+1}) \notag \\
    &\ +\dfrac{1}{\gamma} \big(D(\xx,\xx_k) - D(\xx,\xx_{k+1}) -D(\xx_{k+1},\xx_k) \big)  \notag \\
    &\ + \mu_0 D(\xx_{k+1},\xx_{k})  +\gamma \|\rr_k-\mathsf{R}_k\|^2+ \pair{\hat \xx_{k+1}-\xx}{\mathsf{R}_k-\rr_k},\label{e:oz}
\end{align}
which is \eqref{e:suachua1}.\\
(ii). We next consider the case
 $\theta=1. $
	The term in the first line in \eqref{e:gap} with $x=x^\star,\ v=v^\star$ can be estimated as
	\begin{align}
	&\pair{K(x_{k+1}-y_k)}{v^\star-v_{k+1}}+\scal{K(x_{k+1}-x^\star)}{v_{k+1}-u_k}  \notag\\
	&\quad= \pair{L(\xx_{k+1}-\xx_k)}{\xx_{k+1}-\xx^\star}
	-  \pair{L(\xx_{k}-\xx_{k-1})}{\xx_{k}-\xx^\star}
	-  \pair{L(\xx_k- \xx_{k-1})}{\xx_{k+1}-\xx_k}\notag\\
	&\quad \leq b_{k+1}-  b_k + \dfrac{\|K\|}{2} \big(\|\xx_k- \xx_{k-1})\|^2
	+ \|\xx_{k+1}-\xx_k\|^2\big)\notag \\
	&\quad \le b_{k+1}- b_k+ \|K\|\big(D(\xx_{k},\xx_{k-1}) +D(\xx_{k+1},\xx_k) \big).
		\label{e:dvd0}
	\end{align}
	Using the triangle inequality and the strong convexity of $\varphi$ and $\psi$, we obtain
	\begin{align}
	\dfrac{\mu }{2}\|x_{k+1}-y_k\|^2 + \frac{\nu }{2}\|v_{k+1}-u_k\|^2 &\leq \mu_0 \big(\|\xx_k-\xx_{k+1}\|^2 +\|\xx_k-\xx_{k-1}\|^2 \big)\notag\\
	&\leq 2\mu_0 (D(\xx_{k+1},\xx_{k}) +D(\xx_{k},\xx_{k-1})).
	\label{e:dvd2}
	\end{align}
	We  can estimate the two last term in the third line of \eqref{e:gap} as  
	\begin{align}
	    \pair{x_{k+1}-x}{ \nabla h(y_k)-z_k}+\pair{ \nabla \ell(u_k)-t_k}{v_{k+1}-v} \le \gamma \|\rr_k-\mathsf{R}_k\|^2+ \pair{\hat \xx_{k+1}-\xx}{\mathsf{R}_k-\rr_k}. \label{e:dvd3}
	\end{align}
		Therefore, inserting \eqref{e:dvd0},  \eqref{e:dvd1}, \eqref{e:dvd2} and \eqref{e:dvd3} into 
	\eqref{e:gap}, we get
	\begin{align}
    G(x_{k+1},v^\star)-G(x^\star,v_{k+1}) &\le b_{k+1}-  b_k+ \|K\| \big(D(\xx_{k},\xx_{k-1}) +D(\xx_{k+1},\xx_k) \big) \notag \\
    &\ +\dfrac{1}{\gamma} \big(D(\xx^\star,\xx_k) - D(\xx^\star,\xx_{k+1}) -D(\xx_{k+1},\xx_k) \big)  \notag \\
    &\ +2\mu_0 (D(\xx_{k+1},\xx_{k}) +D(\xx_{k},\xx_{k-1}))\notag \\
    &\ +\gamma \|\rr_k-\mathsf{R}_k\|^2+ \pair{\hat \xx_{k+1}-\xx}{\mathsf{R}_k-\rr_k},
\end{align}
which implies that
\begin{align}
 \gamma [ G(x_{k+1},v^\star)-G(x^\star,v_{k+1})] & \le D(\xx^\star,\xx_k) -\gamma b_k-D(\xx^\star,\xx_{k+1})+\gamma b_{k+1}\notag \\
    &-(1-2\gamma \mu_0-\gamma \|K\|)D(\xx_{k+1},\xx_k) \notag \\
    &+(2\gamma \mu_0 +\gamma  \|K\|)D(\xx_{k},\xx_{k-1}) \notag \\
    &+\gamma^2 \|\rr_k-\mathsf{R}_k\|^2+\gamma \pair{\hat \xx_{k+1}-\xx}{\mathsf{R}_k-\rr_k},\label{inedad10}
\end{align}
which proves \eqref{e:suachua2}.
\end{proof}

The convergence of Agorithm \ref{alg3} depends on the choices of $\theta$
and $(\omega_k)_{k\in\NN}$.  
We first main convergence result can be now stated where we prove the sublinear convergence rate in expeactation of the primal-dual gap for general convex case, i.e. $\alpha=0$.

\begin{theorem}\label{t:2}
Let $(\bar x_s)_{s \in \NN}, \ (\bar v_s)_{s \in \NN}$ be generated by Algorithm \ref{alg3} with $\theta= 1,\omega_k = 1/m$.
Suppose that Assumption 2.5 and Assumption 2.6 are satisfied. Assume that
\begin{align}\label{dkg}
   \begin{cases}
   12\gamma L_1 &< 1 \\
   4\gamma \mu_0+2\gamma \|K\|+8L_2 \gamma^2 &\le 1.
   \end{cases} 
\end{align}
For every $N\in\NN$, define 
\begin{equation}
	\hat{x}_N= \big(\sum_{s=1}^N \bar x_{s}\big)/N; \text{and}\;
	\ \hat {v}_N=\big(\sum_{s=1}^N \bar v_{s}\big)/N.
	\end{equation}
	Then we have 
 \begin{equation}
 \EE [G(\hat{x}_{N},v^\star)-G(x^\star,\hat{v}_{N})] \le \dfrac{D( \xx^\star,\bar \xx_0)+4L_1 \gamma^2(m+2)\big[G(\bar x_0,v^\star)-G(x^\star,\bar v_0)\big]}{m \gamma (1-12L_1 \gamma)N}.
 \end{equation}
 \end{theorem}

\begin{proof}Let $s\in\NN$, consider the stage $s$,
using Lemma \ref{vietlai} \ref{vietlai:ii} with $\alpha=0$, 
we have 
\begin{align}
 \gamma [ G(x_{k+1},v^\star)-G(x^\star,v_{k+1})] & \le D(\xx^\star,\xx_k) -\gamma  b_k-D(\xx^\star,\xx_{k+1})+\gamma b_{k+1}\notag \\
    &-(1-2\gamma \mu_0-\gamma \|K\|)D(\xx_{k+1},\xx_k) \notag \\
    &+(2\gamma \mu_0 +\gamma  \|K\|)D(\xx_{k},\xx_{k-1}) \notag \\
    &+\gamma^2 \|\rr_k-\mathsf{R}_k\|^2+\gamma \pair{\hat \xx_{k+1}-\xx}{\mathsf{R}_k-\rr_k}.\label{e:suachua2a}
\end{align}
Set \begin{equation}
    T_k= G(x_k,v^\star)-G(x^\star,v_k)\geq 0.
\end{equation}
Let us next denote $\xi_k = (i_k,j_k)$ and $\xi_{[k]}$ denote that history $\{\xi_0,\xi_1,\ldots, \xi_{k}\}$.  We also denote by $\E_{\xi_k}$ the conditional 
expectation with respect to $\xi_{[k-1]}$.
It follows that
    \begin{align}
	\gamma \EE_{\xi_k} \big[T_{k+1}\big] 	&\le D(\xx^\star,\xx_k) -\gamma  b_k-\E_{\xi_k} D(\xx^\star,\xx_{k+1})+\gamma \E_{\xi_k} b_{k+1}\notag -(1-2\gamma \mu_0-\gamma \|K\|) \E_{\xi_k}D(\xx_{k+1},\xx_k) \notag \\
    &\ \ +(2\gamma \mu_0 +\gamma \|K\|)D(\xx_{k},\xx_{k-1}) +\gamma^2 \E_{\xi_k}\|\rr_k-\mathsf{R}_k\|^2. \label{e:cv}
    \end{align}
 Moreover, using \eqref{e:var2} in  Corollary \ref{lm1}, 
 \begin{equation}
     \E_{\xi_k}\|\rr_k-\mathsf{R}_k\|^2 \leq 8L_2 D(\xx_k,\xx_{k-1}) + 8L_1 T_k + 4L_1 G(\bar x,v^\star)-G(x^\star,\bar v). \label{ineE}
 \end{equation}
Therefore, we derive from \eqref{e:cv}, \eqref{ineE} and \eqref{dkg} that  
    \begin{align}
 \gamma \EE_{\xi_k} \big[T_{k+1}\big] 
	&\le \big[D(\xx^\star,\xx_k)+(2\gamma\mu_0 +\gamma  \|K\|+ 8L_2 \gamma^2) D(\xx_k,\xx_{k-1})-\gamma b_k \big]\notag \\
	&\ \ -\EE_{\xi_k} \big[D(\xx^\star,\xx_{k+1})+(2\gamma\mu_0 +\gamma \|K\|+8L_2 \gamma^2 ) D(\xx_{k+1},\xx_{k})-\gamma b_{k+1} \big] \notag \\
	&\ \ +8L_1 \gamma^2 T_k+4L_1 \gamma^2 \big(G(\bar x,v^\star)-G(x^\star,\bar v)\big) \notag \\
	&\le e_k-\E_{\xi_k}e_{k+1}+8L_1 \gamma^2 T_k  +4L_1 \gamma^2 \big(G(\bar x_s,v^\star)-G(x^\star,\bar v_s)\big), \label{tkd}
	\end{align}
where we set $$ e_k=D(\xx^\star,\xx_k)+(2\gamma\mu_0  +\gamma \|K\|+8L_2 \gamma^2 ) D(\xx_k,\xx_{k-1})-\gamma b_k.$$
We have 
\begin{align}
    |b_k| &\le \|L\| \| \xx_{k}-\xx_{k-1}\| \|\xx_{k}-\xx^\star\| \notag \\
    &\le \dfrac{\|K\|}{2} \big(\| \xx_{k}-\xx_{k-1}\|^2 +\|\xx_{k}-\xx^\star\|^2 \big). \notag
\end{align}
So the condition \eqref{dkg} implies that $e_k \ge 0 \ \ \forall k \in \NN.$  Taking the expectation with respect to all the history  in the stage $s$ and summing the inequality \eqref{tkd} from $k=0$ to $k=m-1$, we obtain	
\begin{align}
 \gamma \sum \limits_{k=0}^{m-1} \EE [ T_{k+1}] &\le e_0-\E e_m+8L_1 \gamma^2 \sum \limits_{k=0}^{m-1} \EE [T_k] 
+4mL_1 \gamma^2\big[G(\bar x_s,v^\star)-G(x^\star,\bar v_s)\big] \notag \\
 &\le (e_0+8L_1\gamma^2 T_0)-(\E e_m+8L_1 \gamma^2 \E[T_m])+8L_1 \gamma^2 \sum \limits_{k=1}^{m} \EE [T_k] \notag \\
 &\ \  +4mL_1 \gamma^2\big[G(\bar x_s,v^\star)-G(x^\star,\bar v_s)\big].
\end{align}
In turn,
\begin{align*}
   \gamma (1-8L_1 \gamma)\sum \limits_{k=1}^{m} \EE [T_k] &\le (e_0+8L_1\gamma^2 T_0)-(\E e_m+8L_1 \gamma^2 \E[T_m])\notag \\
   &\ \ +4mL_1 \gamma^2\big[G(\bar x_s,v^\star)-G(x^\star,\bar v_s)\big].
\end{align*}
Using the convexity-concavity of $G$, we obtain
\begin{align} \label{dads}
     m \gamma(1-8L_1 \gamma) \E [G(\bar x_{s+1},v^\star)-G(x^\star,\bar v_{s+1})] &\le (e_0+8L_1\gamma^2 T_0)-(\E e_m+8L_1 \gamma^2 \E[T_m]) \notag \\
    & \ \ +4mL_1 \gamma^2\big[G(\bar x_s,v^\star)-G(x^\star,\bar v_s)\big].
\end{align}
Note that by the choices of $(x_{0}^s)_{s\in\NN}$ and $(x_{-1}^s)_{s\in\NN}$, we have  
$$\E e_m+8L_1 \gamma^2 \E[T_m] = \EE e^{s+1}_0+8L_1\gamma^2 \EE T^{s+1}_0.$$
Summing \eqref{dads} from $s=0$ to $s=N-1$, we derive
\begin{align}
    m \gamma (1-8L_1 \gamma)\sum \limits_{s=1}^N \E [G(\bar x_s,v^\star)-G(x^\star,\bar v_s)]& \le (e^0_0+8L_1\gamma^2 T^0_0)-(\E e^{N-1}_m+8L_1 \gamma^2 \E[T^{N-1}_m]) \notag \\
    &+4mL_1 \gamma^2 \sum \limits_{s=0}^{N-1} \E [G(\bar x_s,v^\star)-G(x^\star,\bar v_s)],
\end{align}
which implies that
\begin{align}
   m \gamma (1-12L_1 \gamma)\sum \limits_{s=1}^N \E [G(\bar x_s,v^\star)-G(x^\star,\bar v_s)]  &\le   (e^0_0+8L_1\gamma^2 T^0_0)-(\E e^{N-1}_m+8L_1 \gamma^2 \E[T^{N-1}_m]) \notag \\
   &\ \ +4mL_1 \gamma^2 [G(\bar x_0,v^\star)-G(x^\star,\bar v_0)] \notag \\
   &\le D( \xx^\star,\bar \xx_0)+4L_1 \gamma^2(m+2)\big[G(\bar x_0,v^\star)-G(x^\star,\bar v_0)\big]. \notag
    \end{align}
$$\Rightarrow \E [G(\hat x_N,v^\star)-G(x^\star,\hat v_N)] \le \dfrac{D( \xx^\star,\bar \xx_0)+4L_1 \gamma^2(m+2)\big[G(\bar x_0,v^\star)-G(x^\star,\bar v_0)\big] }{m\gamma (1-12L_1 \gamma)N}.$$
The proof is completed.

\end{proof}

\begin{remark} Recently, there appeared various publications where the convex-concave saddle problems were investigated in stochastic setting; see \cite{Bach1,Chen14,Com02,Dung20,bang1,bang4,Arizona,Nem09,Jud11,Shi17,Zhao19,Wang17} for instances and the reference therein. These existing methods are different 
from our proposed algorithm. Thus
we highlight  several works in which the convergence of the gap functions were investigated for non-strongly convex-concave problems.
\begin{enumerate}
    \item Theorem \ref{t:2} improves the result in \cite[Section 4]{Dung20} 
in which the variance reduction technique as well as the Bregman distances were not exploited. Further connections to existing works \cite{bang1,sva2,Nem09} can be found in \cite{Dung20}.
\item
Theorem \ref{t:2} showed that the ergodic convergence rate in expectation of the primal dual gap is $\mathcal{O}(1/N)$ in which the constant is independent of the problem size $n$. This rate was also achieved in \cite{Arizona}
with a different method but the constant depends on $n$. Note that the rate $\mathcal{O}(1/N)$ is as fast as the several   deterministic methods \cite{te15,Bot2014,Bot2015,Cham16,ar1}.  
\item The rate $\mathcal{O}{(1/\sqrt{N}})$ of the the ergodic convergence rate in expectation of the primal dual gap was obtained in \cite{Nem09,Jud11,Zhao19}.
\end{enumerate}

\end{remark}

We next show that one can improve the rate of the gap function from the sublinear convergence to linear convergence when the $f$ and $g^*$ are strongly convex relative to $\varphi$ and 
$\psi$. 
\begin{theorem}\label{t:shi}
Suppose that Assumptions  \ref{as:1}, \ref{giathiet} and \ref{giathiet3} are satisfied. Choose $M'>\dfrac{2\|K\|}{\alpha}$, denote $\alpha'=\alpha-\frac{2\|K\|}{M'}$, $\tau=1+\gamma \alpha'$, $\eta=4\gamma^2 L_1$, $\sum \limits_{k=1}^{m} \tau^{k-1}=\delta$.
  Assume that 
  \begin{align}0< \gamma< \min \big \{\dfrac{1}{2\|K\|M'+\mu_0},\ \dfrac{-1+\sqrt{1+\frac{\alpha'}{4L_1}}}{\alpha'} \big \} \label{dk1gama} \end{align}
  Let $m$ is a positive integer such that \begin{align} m> \dfrac{\ln \lambda}{\ln \tau}\ \ \text{where} \ \lambda=\dfrac{\gamma-\eta \tau}{\eta}. \end{align}
 Let $(\bar x_s)_{s \in \NN}, \ (\bar v_s)_{s \in \NN}$ be sequences generated by Algorithm \ref{alg3} with $\theta= 0$, $\omega_k = \tau^{k-1}/\delta$.
Then, we have
\begin{align}
    \E [G(\bar x_s,v^\star)-G(x^\star,\bar v_s] \le \dfrac{ \lambda^{-s}}{\eta \delta} [D(\xx^\star,\bar \xx_0)+\eta(1+\delta)(G(\bar x_0,v^\star)-G(x^\star,\bar v_0))]. \label{kq2}
\end{align}
\end{theorem}
\begin{proof}
By our assumption, $\alpha'=\alpha-2\|K\|/M'> 0$. Hence, we can rewrite \eqref{e:suachua1} with 
$M=M'$ as
\begin{align}
\gamma  [G(x_{k+1},v^\star)&-G(x^\star,v_{k+1})]\notag \\
&\le  D(\xx^\star,\xx_{k})-(1+{\gamma \alpha'})D(\xx^\star,\xx_{k+1})-(1-2\gamma \|K\| M'-\gamma \mu_0) D(\xx_{k+1},\xx_k) \notag \\
&+\gamma^2 \|\rr_k-\mathsf{R}_k\|^2+ \gamma  \pair{\hat \xx_{k+1}-\xx}{\mathsf{R}_k-\rr_k}. \label{dgap2d1}
\end{align}
 Taking the conditional expectation with respect to $\xi_{[k-1]}$ both sides of \eqref{dgap2d1}, using the condition \eqref{dk1gama} and Corollary \ref{lm1},  we get:
 \begin{align}
     \gamma \E_{\xi_k}[G(x_{k+1},v^\star)&-G(x^\star,v_{k+1})] \le  D(\xx^\star,\xx_{k})-(1+\gamma \alpha')\E _{\xi_k}D(\xx^\star,\xx_{k+1}) \notag \\
     &\ \  +4\gamma^2 L_1 \bigg([G(x_k,v^\star)-G(x^\star,v_k)]+[G(\bar x_s,v^\star)-G(x^\star,\bar v_s)]\bigg).
     \label{inegd}
 \end{align}
 Denote $T^s_k:=T_k=G(x_k,v^\star)-G(x^\star,v_k),\ \bar T_s=G(\bar x_s,v^\star)-G(x^\star,\bar v_s)$,
 $d^s_k:=d_k:=D(\xx^\star,\xx_k)$.\\
 We have:
 \begin{align}
   (1+\gamma \alpha')\E_{\xi_k} d_{k+1} \le d_k-\gamma \E_{\xi_k} T_{k+1}+4\gamma^2L_1(T_k+\bar T_s). \label{inegd2} 
 \end{align}
Take expectation on both sides of \eqref{inegd2}, we obtain
 \begin{align}
     \tau \E d_{k+1} \le \E d_k-\gamma \E T_{k+1}+\eta \E T_k+\eta \bar T_s.
 \end{align} 
  Multiplying both sides by $\tau^k$
and telescoping from $k = 0$ to $m - 1$, we derive
\begin{align}
    \tau^m \E d_m \le d_0+\eta T_0+\eta \sum_{k=1}^{m-1} \tau^k \E T_k+\eta \bar T_s \sum_{k=0}^{m-1} \tau^k-\gamma \sum_{k=1}^m \tau^{k-1} \E T_k,
\end{align}
which implies
\begin{align}
    \tau^m \E (d_m+\eta T_m)+(\gamma -\eta \tau) \E \sum_{k=1}^m \tau^{k-1} T_k \le d_0+\eta T_0+\eta \sum_{k=1}^m \tau^{k-1} \bar T_s.
\end{align}
From the conditions of $m$ and $\gamma$, we have $\tau^m > \dfrac{\gamma-\eta\tau}{\eta}$, so we obtain:
\begin{align}\label{inegd3}
    \dfrac{\gamma-\eta \tau}{\eta}\E (d_m+\eta T_m)+(\gamma -\eta \tau) \E \sum_{k=1}^m \tau^{k-1} T_k \le d_0+\eta T_0+\eta  \sum_{k=1}^m \tau^{k-1} \bar T_s.
\end{align}
From $G$ is convex-concave, \eqref{inegd3} implies
\begin{align}
    \E [d^{s+1}_0+\eta T^{s+1}_0+\eta \delta \bar T_{s+1}] \le \dfrac{\eta}{\gamma-\eta \tau} [d^s_0+\eta T^s_0+\eta  \delta \bar T_s].
\end{align}
Therefore
\begin{align}
   \E [D(\xx^\star,\xx^{s+1}_0]+\eta (G(x^{s+1}_0,v^\star)&-G(x^\star,v^{s+1}_0))+\eta \delta \bar T_{s+1}] \notag \\
   &\le \dfrac{\eta }{\gamma-\eta  \tau} [D(\xx^\star,\xx^s_0)+\eta (G(x^{s}_0,v^\star)-G(x^\star,v^{s}_0)) +\eta  \delta \bar T_s].
\end{align}
Hence,
\begin{align}
    \eta \delta \E [G(\bar x_s,v^\star)-G(x^\star,\bar v_s] \le \lambda^{-s} [D(\xx^\star,\bar \xx_0)+\eta(1+\delta)(G(\bar x_0,v^\star)-G(x^\star,\bar v_0))].
\end{align}
The proof is completed.
 \end{proof}
\begin{remark} Here are some remarks.
\item[(i)] The choice of $(\omega_k)_{0\leq k\leq m}$ is the same in \cite{Shi17}. However, the resulting algorithm 
is full-splitting and hence it is different from \cite{Shi17}.  Theorem \ref{t:shi} can be viewed a development of \cite[Theorem 1]{Shi17} for our splitting method.
\item[(ii)]  In \eqref{kq2}, $\lambda$ is greater than 1, indeed: $$\lambda>1 \Leftrightarrow \dfrac{1-4\gamma L_1 \tau}{4\gamma L_1}>1 \Leftrightarrow \gamma (\tau+1)<\dfrac{1}{4L_1} \Leftrightarrow \gamma(2+\gamma \alpha')<\dfrac{1}{4L_1}, $$
which is satisfied by \eqref{dk1gama}.
\end{remark}

\begin{remark} Here are some connections to existing works where variance reduction techniques are used.
\begin{enumerate}
    \item Linear convergence in expectation of the primal-dual gap was established in \cite{Shi17} for a different stochastic variance reduce algorithm. 
    \item The authors in \cite{Bach1}
also proposed stochastic variance reduce algorithm for saddle point problems with the linear convergence in expectation of the iterates.
\item In the context of solving empirical composition optimization problem, the
stochastic variance reduced primal dual algorithms with the Euclidean norms  were in \cite{Devraj19} in which the linear convergences of the iterates in expectation were achieved.
\item For a special case of Problem \ref{prob:1} where $\ell =0$ and $f=0$, under additional assumption on the linear operator $L$,  the method proposed in \cite{Du19} with the Euclidean distances achieves the linear convergence rate even when the strongly convex-concave condition is not full-filled.
\end{enumerate}

\end{remark}



\maketitle

\end{document}